\documentclass[12pt]{amsart}

\usepackage{amsmath}

 \usepackage [latin1]{inputenc}

\usepackage{tikz-cd}
\usepackage{graphicx}
\usepackage{hyperref}
\usepackage{amsfonts}
\usepackage{amssymb}

\usepackage{color}
\newcommand{\q}[1]{{\textcolor{red}{#1}}}
\newcommand{\qr}[1]{#1}
\newcommand{\rd}[1]{#1}
\newcommand{\bigast}{\ast}

 \newtheorem{thm}{Theorem}[section]

 \newtheorem{cor}[thm]{Corollary}
 \newtheorem{lem}[thm]{Lemma}
 \newtheorem{prop}[thm]{Proposition}
 
\theoremstyle{definition}%
\theoremstyle{definition}
 \newtheorem{defn}[thm]{Definition}
  \newtheorem{exm}[thm]{Example}

  \theoremstyle{remark}%
 \theoremstyle{remark}
 \newtheorem{rem}[thm]{Remark}
 
\numberwithin{equation}{section}

\newcommand{\CC}{\mathcal{C}}

\newcommand{\cupp}{\smallsmile}

\newcommand{\di}{\mathit{d}}

\newcommand{\R}{\mathbb{R}}
\newcommand{\Z}{\mathbb{Z}}
\newcommand{\last}{\!\ast\!}
\newcommand{\D}{\mathcal{D}}

\newcommand{\eqdef}{\colon\!\!\! =}

\newcommand{\F}{\mathcal{F}}

\newcommand{\dd}{\mathrm{d}}

\newcommand{\im}{\mathop{\mathrm{im}}}

\newcommand{\varHo}{\mathcal{H}}

\newcommand{\Ho}{\mathrm{H}}
\newcommand{\id}{\mathrm{id}}

\begin{document}

\title [Relative cohomology]{Relative \qr{De Rham} cohomology in diffeological spaces}

\author{Enrique   Mac\'ias-Virg\'os}
\author{Reihaneh Mehrabi}


\address{CITMAga, University of Santiago de Compostela, 
15782, 
Spain
}


\email{quique.macias@usc.es, reihaneh.mehrabi@rai.usc.es}

\abstract{
We define two different versions  of the relative De Rham cohomology groups  of a diffeological space. 
Additionally, we study a variant of the Mayer-Vietoris sequence and discuss the existence of a relative cup product.\\
Our approach is categorical, replacing open coverings with generating families and unions and intersections with pull-back and join constructions.
}
}

\maketitle


\section{Introduction}
Diffeological spaces are a generalization of differentiable manifolds, which provides a  framework for non-classical objects   like non-Hausdorff quotients of manifolds \qr{and} spaces of smooth functions. Also, diffeology theory seems to be a convenient setting   in various areas of physics.

A diffeology in the set $X$ is a family $\D$ of so-called {\em plots} ---that is, set maps $\alpha\colon U \to X$ whose domains are open subsets $U\subset \R^n$, $n\geq 0$, ---  satisfying three axioms: covering, locality, and compatibility (see Definition \ref{DEFI-DIFFEOL}).

Many classical constructions can be generalized to this context, \qr{a\-mong them}, the graded algebra $\Omega^\bullet(X)$ of  differential forms, the exterior derivative $\dd$ and the De Rham
cohomology algebra $\Ho^\bullet(X)=\Ho(\Omega^\bullet(X),\dd)$, see Section \ref{DERHAM}. In this paper we study two different versions  of  {\em relative} De Rham cohomology. 

In the setting of topological spaces,  the relative  cohomology groups are defined (see \cite[Section 3.1]{HATCHER} or   \cite[Example 5.4.5]{SPANIER}) for any pair $(X,A)$ with $A\subset X$ a subspace.
In the diffeological setting, a relative De Rham cohomology group $\Ho(X,\alpha)$ was defined, in our previous paper \cite{MAC-MEHR}, for any smooth map $\alpha\colon U \to X$ (see Definition \ref{FIRST-REL}), as the cohomology of the complex $\Omega(X,\alpha)=\ker \alpha^*$,  the kernel of the induced morphism $\alpha^*\colon \Omega(X) \to \Omega(U)$.

However, the relative complex $\Omega(X,\alpha)$ \qr{does not have} an associated long exact sequence in cohomology.
 The \qr{reason} is that, in general,  the morphism $\alpha^*$ is not surjective, \qr{as} an argument based on partitions of unity is missing in the diffeological setting. 
 
In order \qr{to address this issue}, we introduce 
a second relative De Rham complex, denoted by $\varOmega(X,\alpha)$. This \qr{new} complex \qr{mirrors} the analogous one given by Bott-Tu in \cite[Prop. 6.49]{BOTT-TU}.
It  is defined as
$$\varOmega^r(X,\alpha)=\Omega^r(X)\oplus \Omega^{r-1}(\alpha),$$
where $\Omega(\alpha)\subset \Omega(U)$ is the subcomplex of the so-called {\em horizontal forms} 
\qr{as considered} by us in \cite{MAC-MEHR}.  The differential
$\di\colon \varOmega^r(X,\alpha) \to \varOmega^{r+1}(X,\alpha)$
is defined as
$\di(\omega,\theta)=(\dd\omega, \alpha^*\omega-\dd\theta)$.

The cohomology of $(\varOmega(X,\alpha),\di)$ will be denoted by $\varHo(X,\alpha)$.
We then obtain a short exact sequence of complexes 
and the associated long exact  sequence in cohomology (see Equations \eqref{SHORT} and \eqref{LONG}). 

As an auxiliary tool we will  introduce in Section \ref{JOIN} the {\em join} of two arbitrary mooth maps $\alpha\colon U \to X $ and $\beta\colon V \to X$. This {\em join} is a
smooth map $\alpha\last\beta\colon U\last V \to X$ which plays the role of the union of two open sets in the topological setting; actually, it is the categorical push-out of the pull-back of $\alpha$ and $\beta$. 

In Section \ref{GEN-FAM} we study how the De Rham complex
of $X$ is determined by giving a {\em generating family}  $\alpha_1,\dots,\alpha_n$ of the diffeology  (\cite[Art. 1.66]{PATRICK}), see Definition \ref{GENFAM}). 
Actually, $X$ is diffeomorphic to the join $\alpha_1\last\cdots\last \alpha_n$, as we proved in \cite{MR2}.
This is again an indication that generating families play the role of open coverings in the diffeological setting%

\qr{Section \ref{MAYER-VIETORIS-1} is dedicated to constructing a Mayer-Vietoris sequence}
\begin{equation}\label{MAYER-VIETORIS-2-BIS}
\begin{tikzcd}[column sep=12pt]
0\arrow[r]& \Omega(\alpha\last\beta) \arrow[r,"J"]& \Omega(\alpha)\oplus\Omega(\beta) \arrow[r,"\varPi"]& \Omega(p_X)
\end{tikzcd}
\end{equation}
where $\Omega(p_X)$ is the complex of pull-back horizontal forms.
It improves upon the one we defined  in \cite{MAC-MEHR}. 

\qr{When the morphism $\varPi$ in 
\eqref{MAYER-VIETORIS-2-BIS} is surjective, we will say that the Mayer-Vietoris sequence is {\em split}.} In this case we have a long exact sequence in cohomology, see Proposition \ref{LONGEXACTSEQUENCE}. For example, the sequence associated with a nice cover considered by Iwase and Izumida in \cite[Theorem 2.3]{IWASE} is split.

In Section \ref{CP} we introduce a relative cup product for the first type cohomology,
$$
\Ho^r (X,\alpha)\otimes \Ho^s (X,\beta)\rightarrow \Ho^{r+s}(X,
\alpha*\beta),
$$
where $(\alpha,U)$, $(\beta,V)$ are arbitrary plots of the diffeology. This product is achieved
by considering the exterior product of differential forms. 

Finally, we define a skew-symmetric relative cup product
$$
\mathcal{H}^r (X,\alpha)\otimes \mathcal{H}^s (X,\beta)\rightarrow \mathcal{H}^{r+s}(X,
\alpha*\beta)
$$
for the second type cohomology, but only under the condition that the Mayer-Vietoris sequence is split. \qr{The formulae required are intricate and far from obvious, inspired by those  in \cite{MACIAS}}. 

\subsection{Funding}
This work was supported by  Xunta de Galicia ED431C 2019/10 with FEDER funds. The first author was partially supported by MINECO-Spain research project PID2020-MTM-114474GB-I00

\subsection{Acknowledgements}
We thank Daniel Tanr\'e for pointing us reference \cite{RIJKE} and an anonymous referee for many useful remarks.

\section{Diffeological spaces}
The main reference for Diffeology Theory is Iglesias's book \cite{PATRICK}.

A {\em parametrization} on the set $X$  is any set map $\alpha\colon U \to X$, where $U\subset \R^n$ is an open subset of $\R^n$, for some $n\geq 0$ depending on $U$. Often, we shall denote it by $(\alpha,U)$.
\begin{defn}\cite[Art. 1.5]{PATRICK}\label{DEFI-DIFFEOL} 
Let $X$ be a non-empty set. A {\em diffeology} on $X$ is any set $\D$  of parametrizations such that the three following
axioms are satisfied:
\begin{enumerate}
\item
{\em Covering}. Any constant parametrization belongs to $\D$.
\item
{\em Locality}. If the parametrization  $(\alpha,U)$  satisfies that for any $p\in U$ there exists an open neighborhood $V\subset U$ of $p$ such that the restriction $\alpha_{\vert V}$ belongs to $\D$, then $\alpha$ belongs to $\D$.
\item
{\em Compatibility}. If $(\alpha,U)$ belongs to $\D$, and $h\colon V\subset \R^m \to U\subset \R^n$ is a ${\mathcal C}^\infty$-differentiable map, then $(\alpha\circ h, V)$ belongs to $\D$.
\end{enumerate}
\end{defn}

The pair $(X,\D)$ will be called a {\em diffeological space}. Axiom 1 ensures that
$\D\neq\emptyset$;  from Axiom 2,
parametrizations can be glued together; Axiom 3 means that we can change coordinates inside $\D$.

A set map $f\colon (X,\D_X) \to (Y,\D_Y)$ between diffeological spaces is a {\em smooth map} if $f\circ \alpha\in \D_Y$ for all $\alpha\in\D_X$.


\section{De Rham cohomology}\label{DERHAM}
\subsection{Differential forms}We begin by recalling the basic definitions of Cartan calculus and De Rham theory in diffeological spaces, as presented in
\cite[Art. 6.28]{PATRICK}. Later we provide a definition  of the relative cohomology groups of a smooth map (Definition \ref{FIRST-REL}).

\begin{defn}A differential form of degree $r\geq 0$ in the diffeological space $(X,\D)$
is a map $\omega$ which associates to each plot $(\alpha,U)\in \D$ a differential $r$-form $\omega_\alpha\in \Omega^r(U)$ \qr{on} the domain $U\subset \R^m$ of the plot, in such a way that, for any ${\mathcal C}^\infty$ change of coordinates $h\colon V \to U$, it happens that
$\omega_{\alpha\circ h}=h^*\omega_\alpha$.

We denote by $\Omega^r(X,\D)$ or simply by $\Omega^r(X)$ the vector space of $r$-forms in $X$.
\end{defn}
\begin{defn}\label{EXT-PROD}The {\em exterior produc}t of the forms $\omega\in \Omega^r(X)$ and $\mu\in\Omega^s(X)$ is the form $\omega\wedge \mu \in \Omega^{r+s}(X)$ defined as
$$(\omega \wedge \mu)_\alpha=\omega_\alpha \wedge_U \mu_\alpha, \quad (\alpha,U)\in\D,$$
where $ \wedge_U $ is the usual exterior product in $\Omega(U)$.
\end{defn}

Analogously,
\begin{defn} The {\em exterior derivative}
$\dd\colon\Omega^r(X,\D) \to \Omega^{r+1}(X,\D)$ is defined as
$$(\dd_X\omega)_\alpha= \dd_U\omega_\alpha, \quad (\alpha,U)\in\D.$$
\end{defn}

The {\em De Rham cohomology groups} $\Ho(X,\D)$ of the diffeological space $(X,\D)$ are those of the complex $(\Omega^\bullet(X,\D),\dd)$.\\

Any smooth map $f\colon X \to Y$ between diffeological spaces induces a morphism $f^*\colon \Omega(Y)\to \Omega(X)$ given by \qr{{\em the pull-back}}
$(f^*\omega)_\alpha=\omega_{f\circ\alpha}$. In particular, if $\alpha\colon U \to X$ is a plot and $\omega$ is a form on $X$, then $\omega_\alpha=\alpha^*\omega$.

Since
$\dd_X\circ f^* = f^*\circ \dd_Y$,
there is an induced morphism 
$f^*\colon \Ho(Y) \to \Ho(X)$ in cohomology given as
$f^*([\omega])=[f^*\omega]$.

\subsection{Horizontal forms}\label{HORIZ}
For later purposes, we need to recall the notion of {\em horizontal form} that we considered in \cite[Section 2.3]{MAC-MEHR}. The $\alpha$-horizontal forms  are called {\em basic forms} in Iglesias' book \cite[Art. 6.38]{PATRICK} .

If $\alpha\colon U \to X$ is a smooth map between diffeological spaces  (in particular, if $\alpha$ is a plot on $X$), the image of the induced map $\alpha^*\colon \Omega(X) \to \Omega(U)$ is contained in the subcomplex  of {\em $\alpha$-horizontal forms}.

\begin{defn}\label{VERTEQ}
The differential form $\theta\in\Omega(U)$ is {\em $\alpha$-horizontal} when for any changes of coordinates $h,h'\colon V \to U$  such that $\alpha\circ h = \alpha\circ h'$,  \qr{the following condition holds:}
$ h^*\theta = (h')^*\theta$.
\end{defn}

The horizontal forms form  a subcomplex $\Omega(\alpha)$ of $\Omega(U)$ because if $\theta\in \Omega^k(\alpha)$ and $\alpha\circ h=\alpha\circ h'$ then
$$h^*(\dd_U\theta)=\dd_Vh^*\theta=\dd_V(h')^*\theta=(h')^*(\dd_U\theta),$$
hence $\dd_U\theta\in\Omega^{k+1}(\alpha)$.

\begin{lem}\label{EASY}If  $X \stackrel{f}{\to} Y \stackrel{g}{\to} Z$ are smooth maps between diffeological spaces and $\omega\in \Omega(g)\subset \Omega(Y)$ is a $g$-horizontal  form, then $f^*\omega\in \Omega(X)$ is a $(g\circ f)$-horizontal form.
\end{lem}

\subsection{Relative De Rham cohomology}\label{SRDRC}
\qr{The following} is the analogous version of a usual notion of relative cohomology for manifolds, \qr{as in Godbillon's book} \cite{Godbillon}.

   Let $\alpha\colon U\to X$ be a smooth map between diffeological spaces.
   Let $\alpha^*\colon \Omega(X)\to \Omega(U)$ be the induced \qr{morphism of} De Rham complexes.
\begin{defn}\label{FIRST-REL}
The {\em relative cohomology groups} (first version)  $\Ho(X,\alpha)$ of the smooth map $\alpha\colon U \to X$  are those of the complex \qr{$\Omega(X,\alpha)=\ker \alpha^*$},  the kernel of 
$\qr{\alpha^\ast} $.
\end{defn}
Note that if $\alpha^*\omega=0$, then $\alpha^*d\omega=d\alpha^*\omega=0$, hence $\ker\alpha^*$ is a subcomplex of $\Omega(X)$ \qr{and} we have an exact sequence
   \begin{equation*}
   	\begin{tikzcd}[column sep=15pt]
   		0 \arrow[r]&\Omega(X,\alpha)\arrow{r}&\qr{\Omega(X)}\arrow{r}{\alpha^*}&\Omega(U).
   	\end{tikzcd}
   \end{equation*}
\qr{However, this definition has a limitation}: the morphism  $\alpha^*$  is not surjective due to the lack of partitions of unity in diffeological spaces. \qr{As a consequence, it is not possible to construct} a long exact sequence in cohomology.

A different relative De Rham theory   for manifolds was introduced by Bott-Tu in \cite[Prop. 6.49]{BOTT-TU}. We will introduce a similar theory for diffeological spaces \qr{as our second version} of the relative cohomology groups.
\begin{defn}Let $\alpha \colon U \to X$ be a smooth map between diffeological spaces. The {\em relative De Rham complex} (second version) $\varOmega(X,\alpha)$ is defined as
$$\varOmega^r(X,\alpha)=\Omega^r(X) \oplus \Omega^{r-1}(\alpha), \quad r\geq 0,$$
with differential
$\di\colon \varOmega^r(X,\alpha) \to \varOmega^{r+1}(X,\alpha)$
given by
$\di(\omega, \theta)=(\dd\omega, \alpha^*\omega-\dd \theta)$.
\end{defn}

The differential is well defined because $\omega\in\Omega(X)$ implies $\alpha^*\omega\in \Omega(\alpha)$. It is easy to check that $\di\circ\di=0$. We denote by  $\varHo^r(X,\alpha)$ the cohomology groups of $(\varOmega^\bullet(X,\alpha),\di)$. A cohomology class is represented by a closed form  $\omega$ in $X$ which becomes exact when pulled back to $U$. 

\begin{prop}
There is 
the short exact sequence of complexes
\begin{equation}\label{SHORT}
\begin{tikzcd}[column sep=15pt]
0 \arrow[r]&\Omega^{r-1}(\alpha)\arrow[r,"i"]&\varOmega^r(X,\alpha)\arrow[r,"\pi"]&\Omega^r(X)\arrow[r]&0,
\end{tikzcd}
\end{equation}
where  $i(\theta)=(0,\theta)$  and  $\pi(\omega,\theta)=\omega$. This induces a long exact sequence in cohomology:
\begin{equation}\label{LONG}\begin{tikzcd}[column sep=15pt]
\cdots \arrow[r] &\Ho^{r-1}(\alpha)\arrow[r,"i^*"]&\varHo^r(X,\alpha)\arrow[r,"\pi^*"]&\Ho^r(X)\arrow[r,"\alpha^*"]&\Ho^r(\alpha)\arrow[r]&\cdots
\end{tikzcd}
\end{equation}
\end{prop}

Notice that $i$ anti-commutes with the derivatives. It is easy to check that the connecting morphism in \eqref{LONG} is just $\alpha^*$.
\begin{rem}
For a differentiable manifold $X$ and an open subet $i_U\colon U \hookrightarrow X$, the complexes $\Omega(X,U)$ and $\varOmega(X,i_U)$ are not isomorphic but have isomorphic cohomology groups, by a partition of unity argument. 
\end{rem}

\qr{The diffeological De Rham complexes $\Omega(X,\alpha)$ and $\varOmega(X,\alpha)$ are related in the following way}.

\begin{prop}\label{MORPHISM}\label{ISOMORPHISM}

	\qr{We define} a morphism of complexes  $F\colon \Omega(X,\alpha)\to \varOmega(X,\alpha)$ \qr{given by} $F(\omega)=(\omega,0).$
Let  $F^*\colon \Ho(X,\alpha)\to \varHo(X,\alpha)$ be the morphism induced in cohomology.
    If $\alpha^*$ is surjective then $F^*$ is an isomorphism.
\end{prop}
\begin{proof}
	We have $dF(\omega)=d(\omega,0)=(d\omega,\alpha^*\omega)=(d\omega,0)$ because $\omega  \in \qr{\ker}\alpha^*$ hence $dF=Fd$.
Then $F^*([\omega])=[(\omega,0)]$ , $d\omega=0$ and $\alpha^*\omega=0$.

	Firstly \qr{we} show that $F^*$ is injective.
	Let $F^*([\omega])=0$, that is, $[(\omega,0)]=0$, hence if 
 $$(\omega,0)=d(\omega_1,\theta_1)=(d\omega_1,\alpha^*\omega_1-d\theta_1),$$ then $\omega=d\omega_1$.
	The problem is that $\omega_1$ is not $\alpha$-horizontal. Let $\tilde{\theta_1} \in\ \Omega(X)$ such that $\alpha^*\tilde{\theta_1}=\theta_1$.
	Then $\omega=d(\omega_1-d\tilde{\theta_1})$; but 
 $$\alpha^*(\omega_1-d\tilde{\theta_1})=\alpha^*\omega_1-\alpha^*d\tilde{\theta_1}=\alpha^*\omega_1-d\alpha^*\tilde{\theta_1}=\alpha^*\omega_1-d\theta_1=0$$ because $\alpha^*\omega_1=d\theta_1$. Hence $\omega=d(\alpha-\text{horizontal form})$, so $[\omega]=0$.
	
	Now, we show that  $F^*$ is surjective.
	Let $[(\omega,\theta)]\in\ \varHo(X,\alpha)$, that is, $d\omega=0$, $\alpha^*\omega=d\theta$. Let $\tilde{\theta}\in\ \Omega(X)$ \qr{such that} $\alpha^*\tilde{\theta}=\theta$. Then $$(\omega-d\tilde{\theta},0)+d(\tilde{\theta},0)=(\omega-d\tilde{\theta},0)+(\dd\tilde{\theta},\alpha^*\tilde{\theta})=(\omega,\theta),$$ hence 
  $$[(\omega,\theta)]=[(\omega-d\tilde{\theta},0)]=F^*([\omega-d\tilde{\theta}]).$$
	 Notice that 
  $$\alpha^*(\omega-d\tilde{\theta})=\alpha^*\omega-\alpha^*d\tilde{\theta}=\alpha^*\omega-d\theta=0.\qedhere$$
\end{proof}	


\begin{exm}[The circle] The cohomology of the circle $\mathbb S^1$, with its diffeology as a manifold, is the usual De Rham cohomology (see \cite{HEC-MAC-SAN}), that is
$\Ho^0(X)=\langle1\rangle=\R$, $\Ho^1(X)=\langle\mathrm{vol}\rangle=\R$. 
\end{exm}

\begin{exm}[The irrational torus]\label{IRRATORUS}
	Let $a$ be an irrational number. The irrational torus $X=\mathbb T_a$ (\cite{DONA-IGLE},  cf. \cite[Exercise 1.4]{PATRICK})  is the quotient of $U=\mathbb R$ by the equivalence relation  $s\sim t$ iff $s-t=m+n a$ for some integers $m,n\in \mathbb Z$. 
	
	This diffeological space $X$ is the space of leaves of the flow $\mathcal F$ on the torus $\mathbb T^2$ by lines of constant slope $a$. This is an example of a  foliation with dense leaves, which is defined by the closed form $\dd y - a \dd x$.
    
	It has been proved that the {\em basic cohomology} $\Ho(M/\mathcal F)$ of any foliation equals the De Rham cohomology of its space of leaves $M/\mathcal F$ as a diffeological space (\cite{HEC-MAC-SAN}), hence $\Ho(\mathbb T_a)\cong \Ho(\mathbb T^2/\mathcal F)$. 	
    
	On the other hand, it is well known that the {basic cohomology} $\Ho(\mathbb T^2/\mathcal F)$  (\cite[Prop. II.1]{KAC-SERG}) of the \qr{flow} $\mathcal F$   is the cohomology of the abelian Lie algebra $\R$, that is, $\qr{\Ho^0=\R}$, $\qr{\Ho^1=\R}$. 
	\end{exm}

\subsection{Subductions}
Subductions (in the terminology of \cite[Art. 1.47]{PATRICK}) are an important particular example of smooth maps in diffeology. They are surjective maps characterized by the local factorization of the plots on the codomain through the  diffeology of the domain (\cite[Art. 1.43]{PATRICK}).

\begin{thm}[{\cite[Art. 6.39]{PATRICK}}]\label{INJECT}
If $\pi\colon U \to X$ is a subduction  then the induced morphism
$\pi^*\colon \Omega(X) \to \Omega(U)$ is injective. 
\end{thm}

\begin{cor}\label{ISOM}
For a subduction $\pi\colon U \to X$, the complex $\Omega(X)$ is isomorphic to the complex $\Omega(\pi)\subset \Omega(U)$ of $\pi$-horizontal forms.
\end{cor}
\begin{proof}Let $\theta \in \Omega^k(\pi)\subset \Omega^k(U)$ be a $\pi$-horizontal form. We define a form $\omega\in \Omega^k(X)$ as follows: if $\gamma\colon W \to X$ is a plot, and $p\in W$, then $\gamma$ locally factors through $\pi$, say
$\gamma= \pi\circ h$ for some $h\colon W_p \to U$. We define $\omega_\gamma=h^*\theta$ in the neighbourhood $W_p$.

It is easy to check that this definition does not depend on $h$ (because $\theta$ is $\pi$-horizontal), that the form $\omega$ is well defined (that is, it verifies the compatibility with the changes of coordinates), and that $\pi^*\omega=\theta$.

This proves that the image of $\pi^*$ is $\Omega(\pi)$. Hence $\Omega(X)\cong \Omega(\pi)$ by Theorem \ref{INJECT}.
\end{proof}

	Then, \qr{for any subduction} $\pi\colon U\to X$, the first type relative cohomology group is $\Ho(X,\pi)=0$ because $\ker\pi^*=0$.


	On the other hand, the second type \qr{De Rham} complex 
	is isomorphic by $\pi^*$ with the complex
	\begin{equation}\label{COMPLEXSUBD}
		\cdots \to \Omega^r(\pi)\oplus \Omega^{r-1}(\pi) \stackrel{\dd}{\longrightarrow}
		\Omega^{r+1}(\pi)\oplus\Omega^r(\pi) \to \cdots
	\end{equation}
	with differential
$\dd(\omega,\mu)=(\dd\omega,\omega-\dd\mu)$.

\begin{exm}
	Let $X=\mathbb T_a$ be the irrational torus of Example \ref{IRRATORUS} and let  $\pi\colon \R \to X$ be the subduction map for the equivalence relation defining $X$.
	Since $\pi$ is a subduction, the forms in $\Omega(X)$ are the $\pi$-horizontal forms in $M=\R$.
	
	First, the horizontal $0$-forms are the constant functions. \qr{Indeed}, let $f$ be a $0$-form, that is, a real function. We consider the maps $h,h'\colon U=\R \to \R$ given by
	$$
	h(t)=t, \quad
	h'(t)=t+m+n\alpha.
	$$
	Clearly $\pi\circ h=\pi\circ h'$. The condition $h^*f=(h')^*f$ means
	$f\circ h=f\circ h'$,
	that is,
	$f(t)=f(t+m+n\alpha)$, for all $t$.
	Since \qr{the integers} $m,n\in\Z$ are arbitrary and the set
	$$\Delta(\alpha)=\{m+n\alpha\colon m,n\in \Z\}$$
	is dense in $\R$, we obtain that $f$ is a constant function.
	
	Analogously, the $\pi$-horizontal $1$-forms are $\omega=c\,dt$ where $c\in\R$ is a constant.
	Then we can write $\Omega^0(\pi)=\langle 1\rangle_\R\cong\R$ and $\Omega^1(\pi)=\langle dt\rangle_\R\cong\R$,
	and the differential $d\colon \Omega^0\to\Omega^1$ is $d=0$.
	
	We can now compute the two types of relative cohomology groups for $\pi$.
	The first type relative groups are $\Ho(X,\pi)=0$ because they are the cohomology of $\ker\pi^*=0$. 
	For the second type groups $\varHo(\pi)$ we consider the complex \eqref{COMPLEXSUBD}:
	$$\cdots \to\Omega^0(\pi)\oplus 0\stackrel{\dd_0}{\longrightarrow} \Omega^1(\pi)\oplus\Omega^0(\pi)\stackrel{\dd_1}{\longrightarrow}0\oplus\Omega^1(\pi)\to\cdots$$
	with differential 
	$$\dd_0 (c,0) =(0,c), \quad \dd_1(a\,dt,b)=(0, a\,dt).$$
	Then
	$\varHo^0(X,\pi)=\ker \dd_0 =0$
	and
	$\varHo^1(X,\pi)=\ker\dd_1/\im\dd_0 =0$.
\end{exm}

\section{Limits and colimits}\label{JOIN}
\qr{In our previous paper \cite{MR2} we outlined the importance of several categorical constructions in the context of diffeological spaces.}
\subsection{Join}
The category of diffeological spaces and smooth maps
is stable under  limits and colimits (see \cite{B-H}). This allows us to introduce the {\em join} of two maps.

\begin{defn}Let $\alpha\colon U \to X$ and $\beta\colon V \to X$ be two smooth maps with the same codomain $(X,\D)$. 
The {\em join} of $\alpha$ and $\beta$ is a diffeological space $U\last V$ endowed with a smooth map $\alpha\last\beta \colon U\last V \to X$, defined as follows.

We take the disjoint union $U \sqcup V$ with the coproduct diffeology generated by the  diffeologies $\D_U$ and $\D_V$ (see \cite[Art. 1.39]{PATRICK}).
Then introduce the equivalence relation in $U\sqcup V$ generated by
$$u\sim v \text{\ iff\ } u\in U, v\in V, \alpha(u)=\beta(v),$$
and we endow the quotient space $U\last V = (U\sqcup V)/\!\sim$ with the quotient diffeology \cite[Art. 1.50]{PATRICK}. We denote the quotient map by 
$$\CC\colon U\sqcup V \to U\ast V.$$

Finally we define the smooth map $\alpha\last\beta \colon U\last V \to X$ as
$$(\alpha\last\beta)([w])=
\begin{cases}
\alpha(u) &\text{if\ } w=u\in U,\\
\beta(v) &\text{if\ } w=v\in V.
\end{cases}$$
\end{defn}

We have the commutative diagram \eqref{DIAG-QUOT} where  $\alpha_U\eqdef \CC\circ i_U$  and $\beta_V\eqdef \CC\circ i_V$:
\begin{equation}\label{DIAG-QUOT}
\begin{tikzcd}[row sep=18pt]
U \arrow[dr,"\alpha_U"']\arrow[r,"i_{U}"]\arrow[bend right,{ddr},"{\alpha}"']&U\sqcup V\arrow[d,"\CC"] &V\arrow[ld,"\beta_V"]\arrow[l,"i_V"']\arrow[bend left]{ddl}{\beta}\\
&U*V \arrow{d}{\alpha *\beta}&\\
&X&
\end{tikzcd}
\end{equation}

\subsection{Pull-back and push-out}This join of two maps is in fact {\em the push-out} of the {\em pull-back} of the two maps, as it is clear from the following constructions. 

First, one can consider the {\em pull-back} $P$ of the maps $\alpha$ and $\beta$, as we did in \cite{MAC-MEHR}. It is defined as the subspace $P\subset U \times V$ of pairs $(u,v)\in U \times V$ such that $\alpha(u)=\beta(v)$. Then the projections
$p_1\colon U \times V \to U$ and $p_2\colon U \times V \to V$ induce the maps $p_U,p_V$ in the following diagram:
\begin{equation}\label{pullback-1}%
\begin{tikzcd}
P\arrow[d,"p_U"',dashed]\arrow[r,"p_V",dashed]&V\arrow[d,"\beta"]\\
U\arrow[r,"\alpha"]&X.
\end{tikzcd}
\end{equation}
Now we consider the {\em push-out} of the maps $p_U, p_V$. It is given as the space $Q$ obtained by introducing in $U\sqcup V$ the equivalence relation generated by
$u\sim v$  iff  $u\in U, v\in V$, and there exists $w\in P$ such that $p_U(w)=u$ and $p_V(w)=v$. New maps $\alpha_U,\beta_V$ appear, as in the following diagram:
$$\begin{tikzcd}
P\arrow[d,"p_U"']\arrow[r,"p_V"]&V\arrow[d,"\beta_V",dashed]\\
U\arrow[r,"\alpha_U",dashed]&Q.
\end{tikzcd}$$
It is easy to check that $Q=U\last V$, and that $\alpha\last\beta \colon U\last V \to X$ is the map obtained by the universal property of the push-out.

We refer to \cite{RIJKE} for more information about the join construction in other settings.
\subsection{Infinite families}
It is possible to give non-finite versions of all those constructions. For instance, the non-finite pull-back would be defined as follows.

\begin{defn}Let $\F=\{\alpha_\lambda \colon U_\lambda\to X\}_{\lambda\in \Lambda}$ be an arbitrary family of smooth maps with codomain the diffeological space $(X,\mathcal D)$.
The {\em pull-back} or {\em limit} of the family $\F$ is the diffeological space 
$$P=\{(x_\lambda) \in\prod_{\lambda\in\Lambda}X_\lambda \colon \alpha_\lambda(x_\lambda)=\alpha_\mu(x_\mu)\quad\forall\lambda, \mu\in \Lambda \}.$$
We endow $\prod_{\lambda\in\Lambda} X_\lambda$ with the product diffeology (the smallest diffeology making smooth all the projections) and $P$ with the subspace diffeology.
The projections $\pi_\lambda\colon\prod X_\lambda\to X_\lambda$ induce smooth maps $p_\lambda\colon P\to X_\lambda$ such that 
$$\alpha_\lambda\circ p_\lambda=\alpha_\mu \circ p_\mu\quad \forall\lambda,\mu\in \Lambda.$$
\end{defn}
Then the usual universal property of the pull-back holds, that is,
 if $Q$ is a diffeological space and  $q_\lambda\colon Q\to X_\lambda$ are smooth maps such that 
$$\alpha_\lambda\circ q_\lambda=\alpha_\mu\circ q_\mu\quad\forall \lambda, \mu\in \Lambda,$$
then there exists a unique smooth map $f\colon Q\to P$ such that 
$p_\lambda \circ f=q_\lambda$, $\forall \lambda\in \Lambda$, 
see Diagram \eqref{BIG}.
\begin{equation}\label{BIG}
\begin{tikzcd}[row sep=18pt]
{} & Q
\arrow[bend right,swap]{ddl}{q_\lambda}
\arrow[bend left]{ddr}{q_\mu}
\arrow{d}{f} & & \\
& P \arrow{dr}{p_\mu} \arrow{dl}[swap]{p_\lambda} \\
X_\lambda  & & 
X\mu 
\end{tikzcd}
\end{equation}
The following result is required for Lemma \ref{BUILD}.

\begin{lem}\label{SMALLBACK}
	Let $\lambda, \mu\in \Lambda$ and let $P_{\lambda\mu}$ be  the pull-back of $\alpha_\lambda$ and $\alpha_\mu$. Then there exists a unique map $F_{\lambda\mu}\colon P\to P_{\lambda\mu}$ such that 
	$$
		p_\lambda = p_\lambda^\prime \circ F_{\lambda\mu}\qr ,\quad
		p_\mu = p_\mu ^\prime \circ F_{\lambda\mu}
	$$
	(see Diagram \eqref{SMALL}). Notice that $F_{\lambda\mu}$ \qr{is} the projection onto the coordinates $\lambda$ and $\mu$.
	
	\begin{equation}\label{SMALL}
		\begin{tikzcd}[row sep=18pt]
			{} & P
			\arrow[bend right,swap]{ddl}{p_\lambda}
			\arrow[bend left]{ddr}{p_\mu}
			\arrow[dashed]{d}{F_{\lambda\mu}} & & \\
			& P_{\lambda\mu} \arrow{dr}{p_\mu ^\prime} \arrow{dl}[swap]{p_\lambda ^\prime} \\
			U_\lambda  & & 
			U_\mu 
		\end{tikzcd}
	\end{equation}
\end{lem}

\begin{rem}
Analogously we can define the {\em join}  of the family $\F$, that will be denoted by  $\bigast_{\lambda\in \Lambda}\alpha_\lambda$.

\end{rem}

\section{Generating families}\label{GEN-FAM}
\begin{defn}\label{GENFAM}{\cite[Art. 1.66]{PATRICK}}]
A collection 
$\F=\{(\alpha_\lambda,U_\lambda)\}_{\lambda\in\Lambda}$ of plots is a {\em generating family} of the diffeology $\D$ if $\D$ is the least diffeology  on $X$ containing the family $\F$. Equivalently,  any plot of $\D$ locally factors through some element of the family $\F$.
\end{defn}
\qr{Without loss of generality, we will assume that the family  $\F$  contains all constant plots.}

\qr{The following result was proved by us in \cite{MR2}. In short, a diffeological space is the join of any generating family.}

\begin{thm}\label{MT}
	Let $(X,\mathcal D)$ be a diffeological space. A family $\mathcal F=\{\alpha_i\colon U_i\to X\}$ of plots is a generating family for $\mathcal D$ if and only if there is a diffeomorphism $\alpha= \alpha_1\last\cdots\last\alpha_n\colon U=U_1\last\cdots\ast U_n\to X$
	commuting with the maps $\alpha_i$ and the natural maps $j_i\colon U_i\to U$, that is, $\alpha\circ j_i=\alpha_i$ for all $i$.
\end{thm}
\begin{cor}\label{GENER-COR}
If $\{(\alpha_1,U_1), \dots, (\alpha_n,U_n)\}$ is a generating family of the diffeology $\D$ on $X$, then
\begin{enumerate}
\item
 the universal map
$\alpha=\alpha_1\ast\cdots\ast\alpha_n\colon U=U_1\ast\cdots\ast U_n \to X$ is a diffeomorphism (hence a subduction), so 
$$\Ho(X,\D)\cong \Ho(U_1\ast\cdots\ast U_n)=\Ho(\alpha_1\ast\cdots\ast \alpha_n).$$
\item
$\Ho(X,\alpha_1\last\cdots\last \alpha_n)=0$ for the first type relative cohomoloy.
\item
$\varHo(X,\alpha_1\last\cdots\last \alpha_n)=0$ for the second type relative cohomology.
\end{enumerate}
\end{cor}

\begin{rem}The two latter results  easily extend to any infinite generating family.
\end{rem}


\qr{In the next Lemmas we  further explore how the De Rham complex of $X$ is  determined by a generating family.}

\begin{lem}\label{ZERO}
Let $\{(\alpha_\lambda,U_\lambda)\}_{\lambda\in\Lambda}$ be a generating family of plots for the diffeological space $X$. If $\omega\in \Omega(X)$ is a form such that $\alpha_\lambda^*\omega=0$ for all $\lambda\in\Lambda$, then $\omega=0$.
\end{lem}
\begin{proof}Any plot $\gamma \colon W \to X$ locally factors through some $\alpha_\lambda$, say $\gamma= \alpha_\lambda\circ h$ for some change of variables $h\colon W \to U_\lambda$. Then, locally,
$$\omega_\gamma=\gamma^*\omega=h^*\alpha_\lambda^*\omega=h^*0=0.\qedhere$$\end{proof}

\begin{lem}\label{BUILD}
Let $\{(\alpha_\lambda,U_\lambda)\}_{\lambda\in\Lambda}$ be a generating familiy of the diffeological space $X$. Let $P$ be the pullback of those plots, with projections $p_\lambda\colon P \to U_\lambda$ as in Diagram \eqref{pullback-1}. If $\theta_\lambda\in \Omega(U_\lambda)$, $\qr{\lambda\in\Lambda}$, are differential forms of the same degree $r$, such that $p_\lambda^*\theta_\lambda=p_\mu^*\theta_\mu$ for all $\lambda,\mu\in \Lambda$, then there is a unique form $\theta \in \Omega(X)$ such that $\alpha_\lambda^*\theta=\theta_\lambda$ for any $\lambda\in\Lambda$.
\end{lem}

In other words, any form $\omega$ in $X$ is completely determined by the pull-backs $\alpha_\lambda^*\omega$.

\begin{proof}The form $\theta$ will be defined by giving its value $\theta_\gamma$ in any plot $\gamma\colon W \to X$. Let $p\in W$, then $\gamma$ locally factors through some $\alpha_\lambda\colon U_\lambda \to X$, say
$\gamma= \alpha_\lambda\circ h$ for some $h\colon W_p\to U_\lambda$. Then we define
$\theta_\gamma=h^*\theta_\lambda$ on $W_p$.

To prove that it is well defined, assume that there are $\lambda,\mu\in \Lambda$
	such that
	$\alpha_\lambda\circ h_\lambda =\gamma_{\vert W_p}=\alpha_\mu\circ h_\mu$.
	Then there exists $F\colon W_p \to P_{\lambda\mu}$ such that
	$p'_\lambda\circ F =h_\lambda$ and $p'_\mu\circ F=h_\mu$, where $P_{\lambda\mu}$ is the pull-back of $\alpha_\lambda$ and $\alpha_\mu$. By Lemma \ref{SMALLBACK} there exists $F_{\lambda\mu}\colon P \to P_{\lambda\mu}$ such that $p_\lambda=p'_{\lambda}\circ F_{\lambda\mu}$ and $p_\mu=p'_\mu\circ F_{\lambda\mu}$. By hypothesis, we have
	$p_\lambda^*\theta_\lambda = p_\mu^*\theta_\mu$, hence
	$$F_{\lambda\mu}^*(p'_\lambda)^*\theta_\lambda
	=F_{\lambda\mu}^*(p'_\mu)^*\theta_\mu.
	$$
	But $F_{\lambda\mu}$ is a subduction map, so the induced morphism $F_{\lambda\mu}^*$ is injective, by  Theorem \ref{INJECT}. That means that $(p'_\lambda)^*\theta_\lambda=(p'_\mu)^*\theta_\mu$, so
	\begin{align*}
		h^*_\lambda\theta_\lambda=&(p'_\lambda\circ F)^*\theta_\lambda
		=F^*(p'_\lambda)^*\theta_\lambda \\
        =&F^*(p'_\mu)^*\theta_\mu
		=(p'_\mu\circ F)^*\theta_\mu=h^*_\mu\theta_\mu.\qedhere
	\end{align*}

\end{proof}

\begin{cor}\label{DISJ}For the disjoint union $U\sqcup V$,  a form $\theta\in \Omega(U\sqcup V)$ is completely determined by giving the forms $i_U^*\theta\in \Omega(U)$ and $i_V^*\theta \in \Omega(V)$.
\end{cor}
\begin{proof}The pullback of $i_U$ and $i_V$ is $P=\emptyset$.
\end{proof}

\section{Mayer-Vietoris sequences}\label{MAYER-VIETORIS-1}
\qr{Let $\alpha\colon U \to X$ and $\beta\colon V \to X$ be two plots} or, more generally, two smooth maps.
In \cite{MAC-MEHR}, a Mayer-Vietoris sequence
\begin{equation*}
0\to \Omega(X) \to \Omega(\alpha)\oplus \Omega(\beta) \to \Omega(P)
\end{equation*}
was defined, where $P$ is the pullback of $\alpha$ and $\beta$ (see Diagram \ref{pullback-1}).
This sequence failed to be exact on the right, due  to the lack of partitions of unity.


\subsection{A new sequence} We will construct a Mayer-Vietoris sequence
\begin{equation}\label{MAYER-VIETORIS-2}
\begin{tikzcd}[column sep=12pt]
0\arrow[r]& \Omega(\alpha\last\beta) \arrow[r,"J"]& \Omega(\alpha)\oplus\Omega(\beta) \arrow[r,"\varPi"]& \Omega(p_X)
\end{tikzcd}
\end{equation}
where $\Omega(p_X)\subset \Omega(P)$ is the space of $p_X$-horizontal forms, for $p_X=\alpha\circ p_U=\beta\circ p_V\colon P \to X$ (see Diagram \eqref{NEW-DIAG})
\begin{equation}\label{NEW-DIAG}
\begin{tikzcd}[column sep=24pt, row sep=12pt]
&P\arrow[ld,"p_U"']\arrow[d,dashed]\arrow[rd,"p_V"]&\\
U \arrow[dr,"\alpha_U"']\arrow[r,"i_{U}"']\arrow[bend right,{ddr},"{\alpha}"']&U\sqcup V\arrow[d,"\CC"] &V\arrow[ld,"\beta_V"]\arrow[l,"i_V"]\arrow[bend left]{ddl}{\beta}\\
&U*V \arrow{d}{\alpha *\beta}&\\
&X&
\end{tikzcd}
\end{equation}
The morphisms $J,\varPi$ in \eqref{MAYER-VIETORIS-2} are defined as
$$J(\rho)=(\alpha_U^*\rho,\beta_V^*\rho), \quad  \varPi(\theta,\tau)=p_U^*\theta-p_V^*\tau.$$
\begin{prop}\label{SURJECT}
The sequence \eqref{MAYER-VIETORIS-2} is well defined and it is exact.
\end{prop}
\begin{proof}
Step 1: The morphism $J$ is well defined: If $\rho$ is $(\alpha\last\beta)$-horizontal then $(\CC\circ i_U)^*\rho$ is $\alpha$-horizontal, as follows from Lemma \ref{EASY}. Analogously, $(\CC\circ i_V)^*\rho$ is $\beta$-horizontal.

Step 2: That the morphism $\varPi$ is well defined  follows again from Lemma \ref{EASY}.

Step 3: It is an exercise to check that $J$ and $\varPi$ are morphisms of complexes (i.e. they commute with the differentials).

Step 4: That $J$ is injective follows from Lemma \ref{ZERO}, because $\{\CC\circ i_U, \CC\circ i_V\}$ is a generating family of plots for $U\last V$.

Step 5: $\im J \subset \ker \varPi$ because $\varPi\circ J=0$; on the other side, $\ker \varPi \subset \im J$ because given $(\theta,\tau) \in \Omega(\alpha)\oplus \Omega(\beta)$ such that $p_U^*\theta=p_V^*\tau$, there exists $\rho \in \Omega(U\last V)$ such that $\alpha_U^*\rho=\theta$ and $\beta_V^*\rho=\tau$, by Lemma \ref{BUILD}. Here we are using that $\{\CC\circ i_U,\CC\circ i_V\}$ is a generating family for $U\last V$. Moreover, the form $\rho$ is $(\alpha\last\beta)$-horizontal.
\end{proof}

\subsection{Split sequences}
The morphism $\varPi$ in \eqref{MAYER-VIETORIS-2} is not generally surjective.

\begin{defn}
The Mayer-Vietoris sequence is \emph{split} if $\varPi$ in \eqref{MAYER-VIETORIS-2} is surjective.
\end{defn}
\begin{exm}
In a finite dimensional manifold $M$ with the natural diffeology, \qr{the} Mayer-Vietoris sequence for the inclusions of two open sets $U\subset M$ and $V\subset M$ is split.
\end{exm}
\begin{exm} In \cite{IWASE}, Iwase and Izumida considered a {\em nice covering} of a diffeological space as a pair $U,V$ of open subsets of the $\mathcal{D}$-topology
for which there exists a partition of unity $\rho_U,\rho_V$ (\cite[Definition 2.1]{IWASE}). Then the corresponding Mayer-Vietoris sequence is split.\end{exm}

\qr{When the sequence \eqref{MAYER-VIETORIS-2} is split} we can take a (vector space) linear section 
\begin{equation}\label{EQTN}%
S\colon \Omega^r(p_X) \to \Omega^r(\alpha)\oplus\Omega^r(\beta)
\end{equation}
where $\varPi\circ S=\qr{\id}$.
In general the section $S$ is not a morphism of complexes, so we consider $\Delta=\dd S-S\dd$, as in Diagram \eqref{FLECHA}
\begin{equation}\label{FLECHA}
	\begin{tikzcd}[column sep=24pt, row sep=18pt]
		\Omega^r(\alpha)\oplus\Omega^r(\beta) \arrow[d,"d"'] &  \Omega^r(p_X) \arrow[l,"S"']\arrow[d,"d"]\arrow[ld,"\Delta"']\\
		\Omega^{r+1}(\alpha)\oplus\Omega^{r+1}(\beta)	& \arrow[l,"S"] \Omega^{r+1}(p_X)
	\end{tikzcd}
\end{equation}
Here, 
$\Delta \colon \Omega^r(p_X) \to\ \Omega^{r+1}(\alpha)\oplus\Omega^{r+1}(\beta)$
 is a morphism of complexes (up to sign) because
\begin{align*}\dd\Delta&=\dd(\dd S-S\dd)=\dd^2-\dd S\dd=-\dd S\dd,\\
\Delta \dd&=(\dd S-S\dd)\dd=\dd S\dd-S\dd^2=\dd S\dd\q .
\end{align*}
Then \qr{$\Delta$} induces a map in  cohomology. But notice that $\varPi \Delta=0$ because 
$$\varPi(\dd S-S\dd)=\varPi \dd S-\varPi S\dd=\dd\varPi S-\dd=\dd-\dd=0$$ \qr{because} $\varPi S=\qr{\id}$, hence $\mathrm{im}\Delta \subset \Omega^{r+1}(\alpha \ast \beta)$.
So it is more appropriate to write
	\begin{equation}\label{truedelta}
		J\Delta=\dd S-S\dd.
	\end{equation}
\begin{prop}\label{LONGEXACTSEQUENCE}
	If $\varPi$ is surjective, we have a long Mayer-Vietoris exact sequence \qr{in cohomology},	
	\begin{equation}\begin{tikzcd}[column sep=10pt]
			\cdots \arrow[r] &\Ho^r(\alpha\ast\beta)\arrow[r]&\Ho^r(\alpha)\oplus	\Ho^r(\beta \qr)\arrow[r]&\Ho^r(p_X)\arrow[r,"\Delta^*"]&\Ho^{r+1}(\alpha\ast\beta)\arrow[r]&\cdots
		\end{tikzcd}
	\end{equation}
	\qr{where} the connecting morphism is $\Delta^*$.	
\end{prop}

\section{Relative cup product}\label{CP}
\subsection{Absolute cup product} We can define a structure of algebra in the De Rham complex $\Omega^*(X)=\oplus_{r\geq 0}\Omega^r(X)$ by means of the exterior product of forms. Since the exterior derivative $d$ verifies
$$d(\omega\wedge \mu)=(d\omega)\wedge \mu+(-1)^r \omega\wedge (d\mu), \qquad \omega \in \Omega^r(X),\, \mu\in \Omega^s (X),$$
the exterior product induces a structure of algebra in the De Rham cohomology $\qr{\Ho}(X,\D)$. 
\begin{defn}
The morphism
$$\cupp\colon {\rm H^r}(X)\otimes {\rm {\rm H^s}}(X)\to H^{r+s}(X),\quad 
[\omega]\cupp[\mu]=[\omega\wedge\mu],$$
is called the \qr{\em absolute cup product}.
\end{defn}
\subsection{Relative cup product, first type}
\begin{defn}\label{RCP1}
Let $(\alpha,U)$, $(\beta,V)$ be two plots and let $(\alpha\last\beta,U\last V)$ be their join.
The {\em relative cup product}
$$\cupp\colon  \Ho^r(X,\alpha) \otimes \Ho^s(X,\beta) \to \Ho^{r+s}(X,\alpha \last \beta),$$
\qr{for the first type relative cohomology}, 
is defined as
	$[\omega]\cupp[\mu] =[\omega \wedge \mu]$.
\end{defn}
\begin{prop}
	The relative cup product in Definition \ref{RCP1} is well defined.
\end{prop}
\begin{proof}We have differential forms
	$\omega\in\Omega^r(X)$ with $\dd_X\omega=0$ and $\alpha^*\omega=0$,	analogously	$\mu\in\Omega^s(X)$, with $ \dd_X\mu=0$ and $\beta^*\mu=0$.
	
	Clearly $\omega\wedge\mu \in \Omega^{r+s}(X)$ and
	$\dd_X(\omega\wedge\mu)=0$
	
	Denote $$\rho \eqdef (\alpha\last\beta)^*(\omega\wedge\mu)\in\Omega(U\last V).$$ It only remains to show that $\rho=0$. 
	
	Let $\gamma \colon W \to U\last V$ be a plot in $U\last V$. We have to show that $(\rho_\gamma)_p=0$, for any $p\in W$. By definition of subduction diffeology, there is some neighbourhood $W_p\subset W$ such that $\gamma_{\vert W_p}$ factors through $U\sqcup V$, say
	$\gamma_{\vert W_p}=\CC\circ \tilde\gamma$
	where  $\qr{\mathcal C}\colon U\sqcup V\to U\last V$	and $\tilde\gamma\colon W_p \to U\sqcup V$.
	
	Assume for instance that $\tilde\gamma(p)\in U$. Then, by definition of coproduct diffeology, there is some neigbourhood $W'_p\subset W_p$ such that $\gamma$ factors through $U$, that is,
	$\gamma_{\vert W'_p}=\CC\circ i_U\circ \bar\gamma$,
	for some map $\bar\gamma\colon W'_p \to U$.
	
	Then in a neighbourhood $W'_p$ of $p\in W$ we have
	$$(\alpha\last\beta)\circ\gamma=(\alpha\last\beta)\circ\CC\circ i_u\circ \bar\gamma=\alpha\circ\bar\gamma,$$
	hence
	\begin{align*}
		\rho_\gamma&=\gamma^*\rho=\gamma^*(\alpha\last\beta)^*(\omega\wedge\mu)=
		((\alpha\last\beta)\circ\gamma)^*(\omega\wedge\mu)\\
		&=(\alpha\circ\bar\gamma)^*(\omega\wedge\mu)=(\bar\gamma)^*\alpha^*(\omega\wedge\mu)=0,
	\end{align*}
	because $\alpha^*(\omega\wedge\mu)=\alpha^*\omega\wedge\alpha^*\mu$
	and $\alpha^*\omega=0$.
	
	The proof is analogous when $\tilde\gamma(p)\in V$.
\end{proof}

\subsection{Relative cup product, second type} We will now  define a relative cup product \qr{for the second type
cohomology}, when two smooth maps or plots $\alpha\colon U\to X$, $\beta\colon V \to X$, are given. We require the Mayer-Vietoris sequence to split.

If the Mayer-Vietoris sequence is split, then we have the map
$$ K=\qr{\id}-S\varPi \colon \	\Omega^r(\alpha)\oplus\Omega^r(\beta) \to\ 	\Omega^r(\alpha)\oplus\Omega^r(\beta),$$
\qr{where $S$ is the linear section considered in \eqref{EQTN}.}
We have 
$$\varPi  K=\varPi-\varPi S\varPi=\varPi-\varPi=0$$ because  $\varPi S=\qr{\id}$. Hence 
$$ K \colon \ \Omega^r(\alpha)\oplus\Omega^r(\beta)\to \Omega^r(\alpha\ast\beta)  $$
and it is more appropriate to write
\begin{equation}\label{trueK}
	JK=\qr{\id}-S\Pi.
\end{equation}
Moreover
$$KJ=(\qr{\id}-S\varPi)J=J-S\varPi J=J.$$
\qr{Notice that} $K$ is not a morphism of complexes, in fact 
$\dd  K- K\dd=-\Delta\varPi$
because 
\begin{align*}
	\dd K- K\dd=&\dd(\qr{\id}-S\varPi)-(\qr{\id}-S\varPi)\dd\\
	=&\dd-\dd S\varPi-\dd+S\varPi \dd\\
	=&(-\dd S+S\dd)\varPi\\
	=&-\Delta \varPi.
\end{align*}
\begin{defn}\label{DEF2-CUP}
	We define a relative cup product,
	\begin{align*}
		\cupp\colon\mathcal{H}^r (X,\alpha)\otimes \mathcal{H}^s (X,\beta)\rightarrow \mathcal{H}^{r+s}(X,
		\alpha*\beta),
	\end{align*}
	for differentiable maps $\alpha\colon U\to X$ and $\beta\colon V\to  X$, as
	\begin{align}\label{5}
		[(\omega,\theta)]\cupp[\omega',\theta')]=[(\omega\wedge\omega', \xi) ],
	\end{align}
	where 
	$$\xi=\qr{K}(\theta\wedge\alpha^*\omega',\qr{(-1)^\omega}\rd{\beta}^*\omega\wedge\theta')+\qr{\rd{(-1)^{\omega}}}\Delta(p_ U^*\theta\wedge p_ V^*\theta').$$

	In what follows we will verify that this product is well defined
    \qr{and skew-symmetric}.
\end{defn}	
\begin{prop}\label{WELL}    
	The product in \eqref{5} is well defined.
\end{prop}
\begin{proof}
	First, $\dd(\omega\wedge\omega')=0$
	because
	$\dd(\omega\wedge \omega')= \dd\omega\wedge\omega'+(-1)^r \omega\wedge \dd\omega'$, with
	$\dd\omega=0$
	and $\dd\omega'=0$.	
	
	Now we will prove that $(\alpha\last\beta)^{*}(\omega\wedge\omega')=\dd\xi$.
	
	We have 
	$$\alpha_U^*(\alpha\last\beta)^{*}(\omega\wedge\omega')=\alpha^*(\omega\wedge\omega')=\alpha^*\omega\wedge\alpha^*\omega'
	=
	\dd\theta\wedge\alpha^*\omega'=
	\dd(\theta\wedge\alpha^*\omega')$$
	and 
	$$\beta_V^*(\alpha\last\beta)^{*}(\omega\wedge\omega')=\beta^*(\omega\wedge\omega')=\beta^*\omega\wedge\beta^*{\omega'}=
\beta^*\omega\wedge\dd\theta'=\qr{(-1)^\omega}\dd(\beta^*\omega\wedge\theta')\qr,$$ 
\qr{because 
$\dd \beta^*\omega=0$ and $$\dd(\beta^*\omega\wedge\theta')=(-1)^\omega\beta^*\omega\wedge\dd\theta'.$$}
Hence 
	{\begin{equation}\label{UNO}
			J(\alpha\last\beta)^{*}(\omega\wedge\omega')=\dd(\theta\wedge\alpha^*\omega',\qr{(-1)^\omega}\beta^*\omega\wedge\theta').
		\end{equation}
	}
	Now we compute $ J\dd\xi$.
	We have
	\begin{align*}
		J\dd\xi=&\dd J\xi= \dd(JK(\theta\wedge\alpha^*\omega',\qr{(-1)^{\omega}}\rd{\beta}^*\omega\wedge\theta'))\\
		&\quad \qr{+(-1)^{\omega}}\dd J\Delta(p_U^*\theta\wedge p_V^*\theta')\\
		= &\dd(\theta\wedge\alpha^*\omega',\qr{(-1)^{\omega}}\rd{\beta}^*\omega\wedge\theta')-\dd S\varPi(\theta\wedge\alpha^*\omega',\qr{(-1)^{\omega}}\rd{\beta}^*\omega\wedge\theta')\\
		&\quad \qr{+(-1)^{\omega}}\dd(\dd S-S\dd)(p_U^*\theta\wedge p_V^*\theta') \\
		= &\dd(\theta\wedge\alpha^*\omega',\qr{(-1)^\omega}\rd{\beta}^*\omega\wedge\theta')-\dd S(p_U^*\theta\wedge p_U^*\alpha^*\omega'-\qr{(-1)^{\omega}}p_V^*\beta^*\omega\wedge p_V^*\theta')\\
		&\quad \qr{-(-1)^\omega}\dd S(\dd p_U^*\theta\wedge p_V^*\theta'+(-1)^{\theta}p_U^*\theta\wedge\dd p_V^*\theta')\\
		= &\dd(\theta\wedge\alpha^*\omega',\qr{(-1)^\omega}\rd{\beta}^*\omega\wedge\theta')-\dd S(p_U^*\theta\wedge p_U^*\alpha^*\omega'-\qr{(-1)^\omega } p_V^*\beta^*\omega\wedge p_V^*\theta')\\
		&\quad \qr{-(-1)^\omega}\dd S(p_U^*\dd\theta\wedge p_V^*\theta'+(-1)^{\theta}p_U^*\theta\wedge p_V^*\dd \theta')\\
		= &\dd(\theta\wedge\alpha^*\omega',\qr{(-1)^\omega}\beta^*\omega\wedge\theta')-\dd S(p_U^*\theta\wedge p_U^*\alpha^*\omega'
		-\qr{(-1)^\omega} p_V^*\beta^*\omega\wedge p_V^*\theta')\\
		&\qr{-(-1)^{\omega}}\dd S(p_U^*\alpha^*\omega\wedge p_V^*\theta'+(-1)^{\theta}p_U^*\theta\wedge p_V^*\beta^*\omega').
	\end{align*}
	Now, since $\alpha\circ p_U =p_X=\beta\circ p_V$, \qr{it follows that}
	$p_U^*\alpha^* = p_\qr V^*\beta^*$. \qr{Thus, we have}
	$$J\dd\xi=\dd(\theta\wedge\alpha^*\omega',(-1)^{\omega}\beta^*\omega\wedge\theta'),$$
 \qr{like in Equation \eqref{UNO}},
	because the term in
 $$p_U^*\theta\wedge p_U^*\alpha^*\omega' = p_U^*\theta\wedge p_V^*\beta^*\omega'= p_U^*\theta\wedge p_V^*\dd\theta'$$
 \qr{has coefficient}
 $$ (-1)- (-1)^\omega(-1)^\theta=(-1)-(-1)^{\omega+\omega-1}   =0,$$
	while the term
 $$p_V^*\beta^*\omega\wedge p_V^*\theta' = p^*_U\alpha^*\omega\wedge p_V^*\theta'=p_U^*\dd\theta\wedge p_V^*\theta'$$
	has coefficient
 $   \qr{(-1)^\omega} - \qr{(-1)^\omega}=0$.
\end{proof}
\begin{thm}\label{SKSY} The cup product is skew-symmetric, that is
	\begin{equation*}%
		[(\omega',\theta')]\cupp[(\omega,\theta)]=(-1)^{\omega\omega'}[(\omega,\theta)]\cupp[(\omega',\theta')].
	\end{equation*}	
\end{thm}
\begin{proof}
	First we will carefully \qr{review} the notation we are using. Given two smooth maps $\alpha\colon U \to X$ and $\beta\colon V \to X$ we consider the pullback $P$ of $\alpha$ and $\beta$ and the map $p_X=\alpha p_U=\beta p_V$, \qr{see Diagram \eqref{pullback-1}.}
    
	If we consider the pullback $P'$ \qr{of the maps} $\beta$ and $\alpha$ \qr{in reverse order}, there is an obvious \qr{isomorphism} $\phi\colon P	\cong P'$ induced by the \qr{{\em flip} map }
	$$\phi\colon U\times V \to V\times U ,\quad   \phi(u,v)=(v,u).$$
	Notice that $\qr{p_U=p_U'\circ\phi , p_V=p_V'\circ\phi}$ and $\qr{p_X=p'_X\circ \phi}$.
	The induced isomorphism in \qr{the horizontal De Rham complexes}
	$$\qr{\phi^*\colon\Omega(p'_X)  \cong \Omega(p_X)}$$
	sends $\Sigma \xi_i^U\otimes\xi_i^V$ into $\Sigma \xi_i^V\otimes\xi_i^U$\qr . 
	Then, in order to simplify notation, we will consider that $\phi^*=\id$\qr .
  	
	Analogously, there is an obvious isomorphism between the joins $\alpha\last\beta   \cong  \beta\last\alpha$ induced by $U \sqcup V  \cong  V \sqcup U$. Again we shall consider that
	$\Omega(\alpha\last\beta)   \cong  \Omega(\beta\last\alpha)$
	is the identity, for the sake of simplicity.
	
	Now remember  that the pair $(\alpha,\beta)$ is \qr{{\em split}} if the map $\varPi$ in the Mayer-Vietoris sequence \eqref{ESTA} is surjective,	
	\begin{equation}\label{ESTA}
		\begin{tikzcd}[column sep=12pt]
			0\arrow[r]& \Omega^r(\alpha\last\beta) \arrow[r,"J"]& \Omega^r(\alpha)\oplus\Omega^r(\beta) \arrow[r,"\varPi"]& \Omega^r(p_X)
		\end{tikzcd}
	\end{equation}	
	where $\varPi(\theta,\theta')=p_U^*\theta-p_V^*\theta'$.
	In this case there is a linear section $ S\colon \Omega^r(p_X) \to \Omega^r(\alpha)\oplus\Omega^r(\beta)$ where $\varPi\circ S=\qr{\id}.$
	
	Now we have	
	\begin{equation*}
		\begin{tikzcd}
			\Omega^r(\alpha)\oplus\Omega^r(\beta) \arrow[r,"\qr{\phi}"] \arrow[d,"\varPi"] &
			\Omega^r(\beta)\oplus\Omega^r(\alpha)  \arrow[d,"\qr{\varPi'}"']  \\
			\Omega^r(p_X)\arrow[r,equal]  &
			\Omega^r(p_X'). 
		\end{tikzcd}
	\end{equation*}
	where \qr{$\qr{\phi}(\theta,\theta')=(\theta',\theta)$ is \qr{a} flip morphism}. \qr{We have  } 
    \begin{align*}
    \qr{\varPi(\theta,\theta')}&=p_U^*\theta-p_V^*\theta'=\phi^*(p_U')^*theta-\phi^*(p_V')^*\theta'\\
    &=-\phi^*((p_V')^*\theta'-(p_U')^*\theta)=-\phi^*\varPi'(\theta',\theta),
    \end{align*}\
    that is, 
 $\qr{\varPi=-\varPi'\phi}$ \qr{by the convention $\phi^*=\id$ above}.
	Then the Mayer-Vietoris  sequence of $\qr{(\beta,\alpha)}$ admits the linear section 
 $$S'=-\phi S,$$
 because
$$\varPi'S'=-\varPi'\phi S=\varPi S=\qr{\id}.$$
	
	The linear right section $S$ induced a left section
	$\qr K \colon \	\Omega^r(\alpha)\oplus\Omega^r(\beta) \to\ 	\Omega^r(\alpha\ast\beta)$
	where $J\qr K=\qr{\id}-S\varPi$ \qr{(see \eqref{trueK})}.
 Notice that $J\qr KJ=J-S\varPi J=J$\qr , that is\qr , $\qr KJ=\qr{\id}$.
	
 \qr{Analogously} we have 
	$$\qr K' \colon \	\Omega^r(\beta)\oplus\Omega^r(\alpha) \to\ 	\Omega^r(\beta\ast\alpha)$$
	\qr{given} by $J'\qr{K'}=\qr{\id}-S'\varPi'$. In fact we have
 $$\qr K'=\qr K\phi,$$ because
	$$\phi J\qr K'=\qr{\id}-(-\phi S)(-\varPi\phi)=\qr{\id}-\phi S\varPi\phi$$ 
 hence 
 $$J\qr K'=\phi-S\varPi\phi=(\qr{\id}-S\varPi)\phi=J\qr K\phi,$$ so 
 $\qr K'=\qr K\phi$.
	
	Finally we had defined $$\Delta \colon \Omega^r(p_X) \to\ \Omega^{r+1}(\alpha \ast \beta)$$ as $J\Delta=\dd S-S\dd$ \qr{ (see \eqref{truedelta})} .
	
	Analogously, \qr{we define} $$\Delta' \colon \Omega^r(p'_X) \to\ \Omega^{r+1}(\alpha)\oplus\Omega^{r+1}(\beta)$$ by $J'\Delta'=\dd' S'-S'\qr\dd$, where $\dd'=\phi\dd\phi$ \qr{as in} the following diagram:
	\begin{equation*}
		\begin{tikzcd}
			\Omega^r(\alpha)\oplus\Omega^r(\beta) \arrow[d,"\dd"'] \arrow[r,"\phi"] &
			\Omega^r(\beta)\oplus\Omega^r(\alpha) \arrow[d,"\dd'"']  \\
			\Omega^{r+1}(\alpha)\oplus\Omega^{r+1}(\beta)\arrow[r,"\phi"]  &
			\Omega^{r+1}(\beta)\oplus\Omega^{r+1}(\alpha). 
		\end{tikzcd}
	\end{equation*}	
    \qr{Notice that the differential of $\Omega(p_X')$ is that of $\Omega(p_X)$.}
    
 Hence
 $\qr{\Delta'=-\Delta}$
 because
 \begin{align*}
 \phi J \Delta'=&J'\Delta'=\dd' S'-S'\dd=\phi\dd\phi(-\phi S)-(-\phi S)\dd\\
 =&\phi(-\dd S + S\dd)=-\phi(J\Delta)
 \end{align*}

 \qr{Now, since}
\begin{equation}\label{XI}
J\xi=J\qr K(\theta\wedge\alpha^*\omega',(-1)^{\omega}\beta^*\omega\wedge\theta')\qr +(-1)^{\omega}J\Delta(p_U^*\theta\wedge p_V^*\theta')
\end{equation}
	\qr{we have}
	$$J'\xi'=J'\qr K'(\theta'\wedge\beta^*\omega,(-1)^{\omega'}\alpha^*\omega'\wedge\theta)\qr +(-1)^{\omega'}J'\Delta'(p_V^*\theta'\wedge p_U^*\theta)$$ 
	so 
	\begin{align*}
		\phi J\xi'=&\phi J\qr K\phi(\theta'\wedge\beta^*\omega,(-1)^{\omega'}\alpha^*\omega'\wedge\theta)\qr +(-1)^{\omega'}\phi J(-\Delta)(p_{\q V}^*\theta'\wedge p_{\q U}^*\theta)\\
		=&\phi J\qr K((-1)^{\omega'}\alpha^*\omega'\wedge\theta,\theta'\wedge\beta^*\omega)\qr -(-1)^{\omega'}\phi J(\Delta)(p_V^*\theta'\wedge p_U^*\theta)\\ 
  =&\phi J\qr K((-1)^{\omega'}(-1)^{\theta\omega'}\theta\wedge\alpha^*\omega',(-1)^{\theta'\omega}\beta^*\omega\wedge\theta')\\
		&\quad\qr -(-1)^{\omega'}\phi J\Delta((-1)^{\theta\theta'}p_\qr U^*\theta\wedge p_\qr V^*\theta').
\end{align*}
So composing with $\qr{(-1)^{\omega\omega'}\phi}$ we have
 \begin{align*}
(-1)^{\omega\omega'}J\xi'=&JK((-1)^{\omega\omega'}(-1)^{\omega'}(-1)^{\omega'\theta}\theta\wedge\alpha^*\omega',(-1)^{\omega\omega'}(-1)^{\theta'\omega}\beta^*\omega\wedge\theta')\\
&-(-1)^{\omega\omega'}(-1)^{\omega'}J\Delta((-1)^{\theta\theta'}p_\qr U^*\theta\wedge p_\qr V^*\theta'),
\end{align*}
which equals $J\xi$ in Equation \eqref{XI}, after checking the signs:
 \begin{itemize}
 \item
 $\omega\omega'+\omega'+\omega'\theta = \omega'(\omega+1+\omega-1)=2\omega\omega'$,\\
 hence
 $$(-1)^{\omega\omega'}(-1)^{\omega'}(-1)^{\omega'\theta}=1.$$
 \item
 $\omega\omega'+\theta'\omega=\omega(\omega'+\omega'-1)=2\omega\omega'-\omega$\\
 hence
 $$(-1)^{\omega\omega'}(-1)^{\theta'\omega}=(-1)^\omega.$$
 \item
 $1+\omega\omega'+\omega'+\theta\theta'=1+\omega\omega'+\omega'+(\omega-1)(\omega'-1)=2+2\omega\omega'-\omega$\\
 hence
 $$(-1)(-1)^{\omega\omega'}(-1)^{\omega'}(-1)^{\theta\theta'}=(-1)^\omega. \qedhere$$
 \end{itemize}
\end{proof}


\end{document}